\documentclass{amsart}
\usepackage{amssymb,amsmath}
\usepackage{url}
\usepackage{color}
\usepackage{mathrsfs}
\usepackage{hyperref}
\usepackage[backrefs]{amsrefs}

\newtheorem{thm}{Theorem}[section]
\newtheorem{pro}[thm]{Proposition}
\newtheorem{cor}[thm]{Corollary}
\newtheorem{lem}[thm]{Lemma}
\newtheorem{df}[thm]{Definition}
\newtheorem{rem}[thm]{Remark}

\newtheorem{con}[thm]{Conjecture}

\newcommand{\la}{\langle}
\newcommand{\ra}{\rangle}
\newcommand{\N}{\omega}
\newcommand{\R}{\mathbb{R}}

\renewcommand{\iff}{\leftrightarrow}

\newcommand{\mc}{\mathcal}

\newcommand{\inter}{\cap}

\renewcommand{\to}{\rightarrow}

\newcommand{\restrict}{\upharpoonright}
\newcommand{\eps}{\varepsilon}
\newcommand{\nil}{\varnothing}
\renewcommand{\P}{\mathbb{P}}
\newcommand{\E}{\mathbb{E}}

\newcommand{\upto}{\upharpoonright}
\newcommand{\nilstr}{\langle\,\rangle}
\DeclareMathOperator{\wt}{wt}
\DeclareMathOperator{\Member}{\textsc{Member}}
\DeclareMathOperator{\Dim}{dim}

\begin{document}

\title{Martin-L\"of randomness and Galton--Watson processes}
\author{David Diamondstone}
	\address{Department of Mathematics, University of Chicago, Chicago IL 60615}
	\email{ded@math.uchicago.edu}
\author{Bj\o rn Kjos-Hanssen}
	\address{Department of Mathematics, University of Hawai{\textquoteleft}i at M{\=a}noa, Honolulu HI 96822}
	\email{bjoern@math.hawaii.edu}

\begin{abstract}
	The members of Martin-L\"of random closed sets under a distribution studied by Barmpalias et al.\ are exactly the infinite paths through Martin-L\"of random Galton--Watson trees with survival parameter $\frac{2}{3}$. To be such a member, a sufficient condition is to have effective Hausdorff dimension strictly greater than $\gamma=\log_2 \frac{3}{2}$, and a necessary condition is to have effective Hausdorff dimension greater than or equal to $\gamma$.

	\noindent\textbf{Keywords:} random closed sets, computability theory.
\end{abstract}

\maketitle
\tableofcontents
\section{Introduction}

Classical probability theory studies intersection probabilities for random sets. A random set will intersect a given deterministic set if the given set is large, in some sense. Here we study a computable analogue: the question of which real numbers are ``large'' in the sense that they belong to some Martin-L\"of random closed set.

Doug Cenzer put together a group of researchers, including some of his students at the University of Florida, which we will refer to as the Florida group. They introduced algorithmic randomness for closed sets in the paper \cite{BBCDW}. Subsequently Kjos-Hanssen \cite{K:09} used algorithmically random Galton--Watson trees to obtain results on infinite subsets of random sets of integers. Here we show that the distributions studied by the Florida group and by Galton and Watson are actually equivalent, not just classically but in an effective sense.

For $0\le\gamma<1$, let us say that a real $x$ is a $\Member_{\gamma}$ if $x$ belongs to some Martin-L\"of random (ML-random) closed set according to the Galton--Watson distribution (defined below) with survival parameter $p=2^{-\gamma}$. We show that for $p=\frac{2}{3}$, this is equivalent to $x$ being a member of a Martin-L\"of random closed set according to the distribution considered by the Florida group.

In light of this equivalence, we may state that
\begin{enumerate}
	\item[(i)] the Florida group showed that in effect not every $\Member_\gamma$ is ML-random, and
	\item[(ii)] Joseph S.\ Miller and Antonio Mont\'alban showed that every ML-random real is a $\Member_\gamma$; the proof of their result is given in the paper of the Florida group \cite{BBCDW}.
\end{enumerate}
The way to sharpen these results goes via \emph{effective Hausdorff dimension}. Each ML-random real has {effective Hausdorff dimension} equal to one. In Section \ref{sec3} we show that
\begin{enumerate}
	\item[(i$'$)] a $\Member_\gamma$ may have effective Hausdorff dimension strictly less than one, and
	\item[(ii$'$)] every real of sufficiently large effective Hausdorff dimension (where some numbers strictly less than one are ``sufficiently large'') is a $\Member_\gamma$.
\end{enumerate}

\section{Equivalence of two models}\label{Nanjing}

We write $\Omega=2^{<\omega}$, and $2^\omega$, for the sets of finite and infinite strings over $2=\{0,1\}$, respectively.
If $\sigma\in\Omega$ is an initial substring (a prefix) of $\tau\in\Omega$ we write $\sigma\preceq\tau$; similarly $\sigma\prec x$ means that the finite string $\sigma$ is a prefix of the infinite string $x\in 2^\omega$. The length of $\sigma$ is $|\sigma|$. Concatenation of strings $\sigma$, $\tau$ is written $\sigma\tau$ or $\sigma^\frown\tau$, the empty string is written ${\nilstr}$ and strings of length one are written $\la i\ra$ or simply $i$, where $i=0,1$. We use the standard notation $[\sigma]=\{x: \sigma\prec x\}$, and for a set $U\subseteq\Omega$, $[U]^\preceq:=\bigcup_{\sigma\in U} [\sigma]$.
 Let $\mc P$ denote the power set operation, $\mc P(X)=2^X$. For a real number $0\le\gamma<1$, $\lambda_{1,\gamma}$ denotes the distribution with sample space $\mc P(\Omega)$ such that for each $\sigma\in\Omega$,
\[
	\lambda_{1,\gamma}(\{S:\sigma\in S\})=2^{-\gamma},
\]
 and the events $\{S:\sigma\in S\}$ are mutually independent for distinct $\sigma$. Let $\lambda^*_\gamma$ be the distribution with sample space $\mc P(\Omega)$ such that for each $J$, writing $p=2^{-\gamma}$,
\[
	\lambda^*_\gamma(\{S: S\cap \{\sigma0,\sigma 1\}=J\}=
	\begin{cases}
		{1-p} 	&	\text{if } J=\{\sigma 0\} \text{ or }J=\{\sigma 1\}\text{;} \\
		2p-1	&	\text{if }J=\{\sigma 0, \sigma 1\},
	\end{cases}
\]
and the events $\{S: S\cap \{\sigma0,\sigma 1\}=J\}$ are mutually independent for distinct $\sigma$. {The notation $\lambda_{1,\gamma}$ is consistent with earlier usage \cite{K:09} and is also easy to distinguish visually from $\lambda^*_\gamma$. } The closed set $\Gamma_{S}$ determined by $S\subseteq\Omega$ is defined by
\[
	\Gamma_S=\{x\in 2^\omega: (\forall \sigma\prec x)\,\sigma\in S\}.
\]
The \emph{Galton--Watson (GW) distribution for survival parameter $2^{-\gamma}$}, also known as the $(1,\gamma)$-induced distribution \cite{K:09}, and as the distribution of a \emph{percolation limit set} \cite{MP},
 is a distribution $\mathbb P_{1,\gamma}$ on the set of all closed subsets of $2^\omega$ defined by
\[
	\mathbb P_{1,\gamma}(\{\Gamma:\Gamma\in E\})= \lambda_{1,\gamma}\{ S: \Gamma_{S}\in E\}.
\]
Thus, the probability of a property $E$ of a closed subset of $2^\omega$ is the probability according to $\lambda_{1,\gamma}$ that a random subset of $\Omega$ determines a tree whose set of infinite paths has property $E$.
Similarly, let
\[
	\mathbb P^*_\gamma(\{\Gamma: \Gamma\in E\})= \lambda^*_\gamma\{ S: \Gamma_{S}\in E\}.
\]
A \emph{$\Sigma^0_1$ subset of $\mathcal P(\Omega)$} is the image of a $\Sigma^0_1$ subset of $\mathcal P(\omega)=2^\omega$ via an effective isomorphism between $\Omega$ and $\omega$.

\begin{df}
	[Martin-L\"of randomness]
	A set of strings $S\in\mathcal P(\Omega)$ is called \emph{$\lambda_{1,\gamma}$-ML-random} if for each uniformly $\Sigma^0_1$ sequence $\{U_n\}_{n\in\omega}$ of subsets of $\mathcal P(\Omega)$ with $\lambda_{1,\gamma}(U_n)\le 2^{-n}$, we have $S\not\in\bigcap_n U_n$.

	A closed set $\Gamma$ is called \emph{$\P_{1,\gamma}$-ML-random} if $\Gamma=\Gamma_S$ for some $\lambda_{{1,\gamma}}$-ML-random set of strings $S$.

	A set of strings $S\in\mathcal P(\Omega)$ is called \emph{$\lambda^*_{\gamma}$-ML-random} if for each uniformly $\Sigma^0_1$ sequence $\{U_n\}_{n\in\omega}$ of subsets of $\mathcal P(\Omega)$ with $\lambda^*_{\gamma}(U_n)\le 2^{-n}$, we have $S\not\in\bigcap_n U_n$.

	A closed set $\Gamma$ is called \emph{$\P^{*}_{\gamma}$-ML-random} if $\Gamma=\Gamma_S$ for some $\lambda^{*}_{\gamma}$-ML-random set of strings $S$.
\end{df}

\begin{lem}
	[Axon \cite{Axon}]
	\label{Logan}

	Let $2^{-\gamma}=\frac{2}{3}$. A closed set $\Gamma\subseteq 2^\omega$ is $\P^*_\gamma$-ML-random if and only if $\Gamma$ is a Martin-L\"of random closed set under the Florida distribution.
\end{lem}

A probability space $(M,\mc M,\mu)$ consists of a set $M$, a $\sigma$-algebra $\mc M$ on $M$, and a measure $\mu$ defined on each set in $\mc M$. For each probability space there is a unique canonical $M$-valued random variable $X$ such that for any $A\in\mc M$, the probability that $X\in A$ is $\mu(A)$. In this way, for $\mu=\lambda_{1,\gamma}$ or $\mu=\lambda^{*}_{\gamma}$ we get the random variable $S\in\mc P(\Omega)$. Conversely, if $Y=f(X)$ is a random variable defined deterministically from $X$ then there is a unique canonical measure $\nu$ such that $\nu(A)$ is the probability that $Y\in A$, i.e. $\nu(A)=\mu(\{x:f(x)\in A\})$.

From such a random variable $S$ we then define further random variables
\begin{alignat*}{2}
	& G_n		&=	&\{\sigma: |\sigma|=n \And (\forall\tau\preceq\sigma)\,\,\tau\in S\},\\
	& G			&=	&\bigcup_{n=0}^\infty G_n, \quad \text{and}\\
	& G_\infty	&=	&\{\sigma\in G: \{\tau\in G : \sigma\prec\tau\}\text{ is infinite}\}.
\end{alignat*}

We have $\Gamma_G=\Gamma_S$ and
$G_\infty\subseteq G\subseteq S$, and values of $G_\infty$ are in one-to-one correspondence with values of $\Gamma_S$.

A value of the random variable $G$ is called a \emph{GW-tree} or a \emph{Florida tree} when $S$ is the canonical random variable associated with $\lambda_{1,\gamma}$ or $\lambda^*_\gamma$, respectively.

Let $e$ be the extinction probability of a GW-tree with parameter $p=2^{-\gamma}$,
\[
	e=\P_{1,\gamma} (\nil) = \lambda_{1,\gamma}(\{S: \Gamma_S=\nil\}).
\]
For any number $a$ let $\overline a=1-a$.
\begin{lem}
	\label{extinct}
	\[
	e=\overline p/p.
	\]
\end{lem}
\begin{proof}
	Note that $\Gamma_{S}=\nil$ iff either (1) ${\nilstr}\not\in S$, or (2) ${\nilstr}\in S$ but
	\[
	\Gamma_{S\cap \{\sigma:\la i\ra\preceq\sigma\text{ or }\sigma\preceq\la i\ra\}}=\nil
	\]
	for both $i\in\{0,1\}$. This gives the equation $e=\overline p + pe^2$.
\end{proof}

We use standard notation for conditional probability,
\[
	\P(E\mid F)=\frac{\P(E\cap F)}{\P(F)};
\]
in measure notation we may also write $\lambda(E\mid F)=\lambda(E\cap F)/\lambda(F)$.

The following lemma is the first indication that there is a connection between GW-trees and Florida trees. We write $\mathbf 1_A$ for the characteristic function of an event or a set $A$, i.e., $\mathbf 1_A=1$ if $A$ occurs, otherwise $\mathbf 1_A=0$.
\begin{lem}\label{301}
	For all $J\subseteq \{\la 0\ra, \la 1\ra\}$,
	\[
	\lambda_{1,\gamma}\left\{G_\infty\cap\{\la 0\ra,\la 1\ra\right\}=J \mid G_\infty\ne\nil \}=\lambda^*_\gamma[G_1=J].
	\]
\end{lem}
\begin{proof}
	By definition, $\lambda^*_\gamma[G_1=J]$ equals
	\[
		(2p-1)\cdot\mathbf 1_{J=\{\la 0\ra,\la 1\ra\}}+\sum_{i=0}^1 (1-p)\cdot\mathbf 1_{J=\{\la i\ra\}},
	\]
	so we only need to calculate $\lambda_{1,\gamma}\left\{G_\infty\cap\{\la 0\ra,\la 1\ra\right\}=J \mid G_\infty\ne\nil \}$. By symmetry, and because the probability that $G_1=\nil$ is $0$, it suffices to calculate this probability for $J=\{\la 0\ra,\la 1\ra\}$. Now if $G_1=\{\la 0\ra, \la 1\ra\}$ then ${\nilstr}$ survives and both immediate extensions are non-extinct. Thus the conditional probability that $G_1=\{\la 0\ra,\la 1\ra\}$ is $\frac{p(1-e)^2}{1-e}=p(1-e)$. By Lemma \ref{extinct}, this is equal to $2p-1$.
\end{proof}

Let a measure $\lambda_{c}$ on $\mc P(\Omega)$ be defined by
\begin{alignat*}{3}
	&\lambda_{c}(M)={}&\lambda_{1,\gamma}&(M \mid G_\infty\ne\nil),&\quad\text{and}
	\\
	&\lambda_i(M)={}&\nu&(M\mid G_{\infty}\ne\nil),&
\end{alignat*}
where $\nu$ is the canonical measure obtained from $G_\infty$.

Let $\mu_i$ and $\mu_c$ be the canonical measures obtained from $G$ when $S$ is the canonical random variable obtained from $\lambda_i$ and $\lambda_c$, respectively (so $\mu_i=\lambda_i$).

Let $\lambda_f$ be the distribution with sample space $\mc P(\Omega)$ such that for each $\sigma\in\Omega$,
\[
	\lambda_f(\{S:\sigma\in S\})=1-p,
\]
 and the events $\{S:\sigma\in S\}$ are mutually independent for distinct $\sigma$. Note that this is exactly the definition of $\lambda_{1,\gamma}$, but with $1-p$ in place of $p$. If the random variable $G$ is defined as before, but on this new space with measure $\lambda_f$, then $G$ again is a GW-tree, but now with survival parameter $1-p \leq \frac{1}{2}$. It turns out that the extinction probability $e'$ of such a tree is 1, so such a tree is almost surely finite.

\begin{lem}
	\[
		e'=\lambda_f(\{S \mid \Gamma_S = \nil \})=1.
	\]
\end{lem}

\begin{proof}
	As in Lemma \ref{extinct}, we observe that $\Gamma_S = \nil$ iff either (1) ${\nilstr}\not\in S$, or (2) ${\nilstr}\in S$ but
	\[
		\Gamma_{S\cap \{\sigma:\la i\ra\preceq\sigma\text{ or }\sigma\preceq\la i\ra\}}=\nil
	\]
	for both $i\in\{0,1\}$. This gives the equation $e'=\overline (1-p) + (1-p)(e')^2$, which has solutions 1 and $\frac{p}{1-p}$. Since $\frac{p}{1-p}>1$, it cannot represent a probability, so $e'=1$.
\end{proof}

\begin{cor}
	\label{no_almost_paths}
	If $S \subseteq \Omega$ is chosen randomly with respect to $\lambda_f$, then for all reals $f$, there are infinitely many initial segments of $f$ which are not in $S$.
\end{cor}

\begin{proof}
	Let $M \subseteq 2^\Omega$ have measure 0, and let $T$ be a fixed, finite set of strings. Define $M' = \{S \mid (\exists S' \in M) S \setminus T = S' \setminus T \}$. Then $\lambda_f(M') \leq \left( \frac{1}{p}\right)^{|T|} \lambda_f(M)=0$, so $M'$ also has measure 0. Now let $T$ vary, and let
	\[
		M''= \{S \mid (\exists S' \in M)(\exists T \subset \Omega, |T|<\infty) S \setminus T = S' \setminus T \}.
	\]
	Then $M''$ is a countable union of measure 0 sets, so has measure 0. If
	\[
		M=\{S \mid \Gamma_S \neq \nil \},
	\]
	then $M''$ is the set we are interested in, which therefore has measure 0.
\end{proof}

\begin{rem}
	\label{david remark}
	Take $M=\{S \mid \Gamma_S \neq \nil \}$, and fix $T$. Then the set $M'$ above is a $\Pi^0_1$ class of measure 0, and is therefore contained in the intersection of the universal ML-test. Since this is true for all finite $T$, the set $M''$ is also contained in the intersection of the universal ML-test.
\end{rem}

We define a $\mu_i\times\lambda_f\to \mu_c$ measure-preserving map $\psi:2^\Omega\times 2^\Omega\to 2^\Omega$. The idea is to overlay two sets $G_i$, $S_f$, so that $G_i$ specifies the extendible nodes of a tree, and $S_f$ specifies the non-extendible nodes. Let $\psi$ be defined by
\[
	\psi(G_i, S_f)=\{\sigma : (\forall \tau \preceq \sigma) \tau \in G_i \cup S_f\}.
\]
(In other words, since $G_i \cup S_f$ will not necessarily be a tree, we take the set of strings in that set whose predecessors are also all in the set to get the largest possible tree contained in that set of strings.)

\begin{lem}
	\label{previous}
	Write $\sigma 2=\{\sigma0,\sigma 1\}$. The following identities hold for every string $\sigma$, and every set $D \subseteq \{\sigma 0, \sigma 1\}$:
	\begin{alignat*}{1}
		(\mu_i \times \lambda_f)(G \cap \sigma 2 = D \mid \sigma \in G)								&=\mu_c(G_\infty \cap \sigma 2 = D \mid \sigma \in G_\infty),\\
		(\mu_i \times \lambda_f)(\psi(G,S) \cap \sigma 2 = D \mid \sigma \in G) 					&=\mu_c(G \cap \sigma 2 = D \mid \sigma \in G_\infty),\text{ and}\\
		(\mu_i \times \lambda_f)(\psi(G,S) \cap \sigma 2 = D \mid \sigma \in \psi(G,S) \setminus G)	&=\mu_c(G \cap \sigma 2 = D \mid \sigma \in G \setminus G_\infty).
	\end{alignat*}
\end{lem}

\begin{proof}
	Note that the event $\sigma \in G_\infty$ implies the event $G_\infty \neq \nil$, and the event $\sigma \in G \setminus G_\infty$ implies that any further events cannot affect the probability of the event $G_\infty \neq \nil$. Thus we may replace $\mu_c$ by $\lambda_{1,\gamma}$ in the above, as $\mu_c$ is $\lambda_{1,\gamma}$ conditioned on the event $G_\infty \neq \nil$. The rest is a straightforward computation, and is omitted.
\end{proof}

\begin{thm}
	The map $\psi$ is $\mu_i \times \lambda_f \to \mu_c$ measure preserving.
\end{thm}

\begin{proof}
	To show that $\psi$ is measure preserving, it suffices to show that it is measure preserving on the basic open sets. We will write $\Omega_n=2^{<n}$ for the set of strings of length less than $n$, and given $S \subseteq\Omega$, we will write $S \upto n$ for $S \cap {\Omega_n}$, the set of strings in $S$ of length less than $n$. Recall that the basic open sets in $2^\Omega$ are the sets of the form $N_T=\{S \mid S \upto n = T \}$ for fixed $n \in \omega$, $T \subseteq {\Omega_n}$. Thus we must show that $\mu_c(N_T)=(\mu_i \times \lambda_f)(\psi^{-1}(N_T))$ for each $T$.

	Since $\mu_c(\{S \mid S \text{ is not a tree}\})=0$, and $\psi(G,S)$ is always a tree, we can ignore elements of $2^\Omega$ which are not trees, and focus our attention on the sets $N_T$ where $T$ is a tree. For $T' \subseteq T$, let $N_{T,T'}=\{S \mid S \text{ is a tree, }S \upto n = T \text{, and }S_\infty \upto n = T' \}$. Observe that $N_T$ is equal to the disjoint union $N_T = \bigcup_{T' \subseteq T} N_{T,T'}$. Thus it suffices to show that $\mu_c(N_{T,T'})=(\mu_i \times \lambda_f)(\psi^{-1}(N_{T,T'}))$ for each pair $T,T'$.

	Suppose $\psi(G,S) \in N_{T,T'}$. Then either $G \upto n = T'$, the extendible tree $G$ contains a non-extendible node, or the tree $\psi(G,S)$ contains an extendible node outside of $G$. The latter event implies that there is a string $f$ such that all but finitely many initial segments of $f$ are elements of $S$. By Corollary \ref{no_almost_paths}, such an event has $\lambda_f$-measure 0, so (up to measure 0 events), $\psi(G,S) \in N_{T,T'}$ implies that $G \upto n = T'$.

	We will prove, by induction, for each $n$, and all $T' \subseteq T \subseteq {\Omega_n}$, that $\mu_c(N_{T,T'})=(\mu_i \times \lambda_f)(\psi^{-1}(N_{T,T'}))$. First, we see that when $n=1$, we either have $T=\nil$ or $T=\{{\nilstr}\}$, and similarly with $T'$. We have that $\mu_c$ is conditioned on the event $G_\infty \neq \nil$, and $\mu_i$ is the distribution of a nonempty extendible tree, so both sides are equal to 1 when $T=T'=\{{\nilstr}\}$, and 0 when $T'=\nil$, so we have equality when $n=1$.

	Now assume $n>0$, and equality holds for $T' \subseteq T \subseteq {\Omega_n}$. Let $U' \subseteq U \subseteq {\Omega_{n+1}}$ with $U \upto n = T$, and $U' \upto n = T'$. We may assume that $U,U'$ are trees. We wish to show that $\mu_c(N_{U,U'})=(\mu_i \times \lambda_f)(\psi^{-1}(N_{U,U'}))$, given that the same equality holds with $T,T'$ in place of $U,U'$. Note that, given $N_{T,T'}$, the event $N_{U,U'}$ may be thought of as the intersection, over all $\sigma \in T \cap 2^{n-1}$, of the events $G \cap \{\sigma 0, \sigma 1\} = U \cap \{\sigma 0, \sigma 1\}$ and $G_\infty \cap \{\sigma 0, \sigma 1\} = U' \cap \{\sigma 0, \sigma 1\}$, and these events are independent for distinct $\sigma$. Similarly, given $\psi^{-1}(N_{T,T'})$, the event $\psi^{-1}(N_{U,U'})$ may be thought (up to events of measure 0) of as the intersection, over all $\sigma \in T \cap 2^{n-1}$, of the events $\psi(G,S) \cap \{\sigma 0, \sigma 1\} = U \cap \{\sigma 0, \sigma 1\}$ and $G \cap \{\sigma 0, \sigma 1\} = U' \cap \{\sigma 0, \sigma 1\}$, and these events are independent for different $\sigma$. By Lemma \ref{previous}, the corresponding probabilities are all equal, so $\mu_c(N_{U,U'} \mid N_{T,T'})=(\mu_i \times \lambda_f)(\psi^{-1}(N_{U,U'}) \mid \psi^{-1}(N_{S,S'}))$. By induction, $\psi$ is measure-preserving.
\end{proof}

Intuitively, a $\lambda_i$-ML-random tree may by van Lambalgen's theorem be extended to a $\lambda_c$-ML-random tree by adding finite pieces randomly; we verify this intuition in the following theorem.

\begin{thm}\label{random}
	For each ML-random Florida tree $H$ there is a ML-random GW-tree $G$ with $G_\infty=H_\infty$.
\end{thm}
\begin{proof}
	Let $H$ be an ML-random Florida tree (i.e., let it be ML-random with respect to the measure $\mu_i$). Let $S \subseteq \Omega$ be ML-random relative to $H$ with respect to the measure $\lambda_f$. Then, by a suitably generalized version of van Lambalgen's theorem\footnote{To be precise, van Lambalgen's theorem holds in the unit interval $[0,1]$ with Lebesgue measure $\lambda$, or equivalently the space $2^\omega$. If $(X,\mu)$ is a measure space then using the measure-preserving map $\varphi:(X,\mu)\to ([0,1],\lambda)$ induced from the Carath\'eodory measure algebra isomorphism theorem \cite{KNe}, we may apply van Lambalgen's theorem as desired.}, $(H,S)$ is ML-random relative to the measure $\mu_i \times \lambda_f$. Now since $\psi$ is effectively continuous, if $(U_n)$ is a uniformly $\Sigma^0_1$ sequence, than so is $(\psi^{-1}(U_n))$. Furthermore, since $\psi$ is $\mu_i \times \lambda_f \to \mu_c$ measure preserving, we have $\mu_c(U_n)=(\mu_i \times \lambda_f)(\psi^{-1}(U_n)$, so $\psi^{-1}$ pulls back $\mu_c$-ML-tests to $\mu_i \times \lambda_f$-ML-tests. Thus, since $(H,S)$ passes every ML-test, so also must $G=\psi(H,S)$, so $G$ is a ML-random GW-tree.

	Now, by Corollary \ref{no_almost_paths} and Remark \ref{david remark}, the set
	\[
		M=\{S' \mid S' \text{ contains all but finitely many initial segments of some real}\}
	\]
	has measure 0, and is contained in the universal $\lambda_f$-ML-test. Since $S$ is ML-random with respect to $\lambda_f$, for any path $f \in \Gamma_G$, there are infinitely many initial segments of $f$ which are not in the set $S$. Thus $f$ must contain infinitely many initial segments in $H$, which means that $f \in \Gamma_H$. Therefore $\Gamma_G \subseteq \Gamma_H$, and $G_\infty \subseteq H_\infty$. But $H \subseteq G$, so we have equality: $G_\infty = H_\infty$.
\end{proof}

We next prove that the live part of every infinite ML-random GW-tree is an ML-random Florida tree.

 \begin{thm}
	\label{thought}
	For each $S$, if $S$ is $\lambda_{1,\gamma}$-ML-random then $G_\infty$ is $\lambda^*_\gamma$-random.
\end{thm}
\begin{proof}
	Suppose $\{U_n\}_{n\in\omega}$ is a $\lambda^*_\gamma$-ML-test with $G_\infty\in\bigcap_n U_n$. Let $\Upsilon_n=\{S: G_\infty\in U_n\}$. By Lemma \ref{301}, $\lambda_{1,\gamma}(\Upsilon_n)=\lambda^*_\gamma(U_n)$. Unfortunately, $\Upsilon_n$ is not a $\Sigma^0_1$ class, but we can approximate it.  While we cannot know if a tree will end up being infinite, we can make a guess that will usually be correct.

	Let $e$ be the probability of extinction for a GW-tree. By Lemma \ref{extinct} we have $e=\frac{\overline p}{p}$, so since $p>1/2$, $e<1$. Thus there is a computable function $(n,\ell)\mapsto m_{n,\ell}$ such that for all $n$ and $\ell$, $m=m_{n,\ell}$ is so large that $e^m\le 2^{-n} 2^{-2\ell}$. Let $\Phi$ be a Turing reduction so that $\Phi^G({n,\ell})$, if defined, is the least $L$ such that all the $2^\ell$ strings of length $\ell$ either are not on $G$, or have no descendants on $G$ at level $L$, or have at least $m_{n,\ell}$ many such descendants.  Let
	\[
		W_n=\{S: \text{ for some $\ell$, $\Phi^G({n,\ell})$ is undefined}\}.
	\]

 	Let $A_G(\ell)=G_\infty\cap \{0,1\}^{\le \ell}$ be $G_\infty$ up to level $\ell$. Let the approximation $A_G(\ell,L)$ to $A_G(\ell)$ consist of the nodes of $G$ at level $\ell$ that have descendants at level $L$. Let
	\begin{align*}
		V_n=	&\{S: A_G(\ell,L)\in U_n \text{ for some $\ell$, where $L=\Phi^G({n,\ell})$}\},\text{ and} \\
		X_n=	&\{S:\text{for some $\ell$, $L=\Phi^G({n,\ell})$ is defined and $A_G(\ell,L)\ne A_G(\ell)$} \}.
	\end{align*}
	Note that $\Upsilon_n =\{S: \text{for some $\ell$, $A_G(\ell)\in U_n$ }\},$ hence $\Upsilon_n\subseteq W_n\cup X_n\cup V_n$. Thus it suffices to show that $\cap_n V_n$, $W_n$, $\cap_n X_n$ are all $\lambda_{1,\gamma}$-ML-null sets.

	\begin{lem}
		$\lambda_{1,\gamma}(W_n)=0$.
	\end{lem}
	\begin{proof}
		If $\Phi(\ell)$ is undefined then there is no $L$, which means that for the fixed set of strings on $G$ at level $\ell$, they do not all either die out or reach $m$ many extensions. But eventually this must happen, so $L$ must exist.

		Indeed, fix any string $\sigma$ on $G$ at level $\ell$. Let $k$ be the largest number of descendants that $\sigma$ has at infinitely many levels $L>\ell$. If $k>0$ then with probability 1, above each level there is another level where actually $k+1$ many
descendants are achieved. So we conclude that either $k=0$ or $k$ does not exist.
	\end{proof}

	From basic computability theory, $W_n$ is a $\Sigma^0_2$ class. Hence each $W_n$ is a Martin-L\"of null set.

	\begin{lem}
		$\lambda_{1,\gamma}(X_n)\le 2^{-n}$.
	\end{lem}
	\begin{proof}

		Let $E_\sigma$ denote the event that all extensions of $\sigma$ on level $L$ are \emph{dead}, i.e. not in $G_\infty$. Let $F_\sigma$ denote the event that $\sigma$ has at least $m$ many descendants on $G(L)$.

 		If $A_G(\ell,L)\ne A_G(\ell)$ then some $\sigma\in\{0,1\}^\ell\cap G$ has at least $m$ many descendants at level $L$, all of which are dead. If a node $\sigma$ has at least $m$ descendants, then the chance that all of these are dead, given that they are on $G$ at level $L$, is at most $e^m$ (the eventual extinction of one is independent of that of another), hence writing $\P=\lambda_{1,\gamma}$, we have
		\begin{align*}
			\mathbb P\{A_G(\ell,L)\ne A_G(\ell)\} \le  \sum_{\sigma\in\{0,1\}^\ell} \mathbb P\{ E_\sigma\And F_\sigma\}= \sum_{\sigma\in\{0,1\}^\ell} \mathbb P\{  E_\sigma \mid F_\sigma \}\cdot \mathbb P\{ F_\sigma \} \\
			\le \sum_{\sigma\in\{0,1\}^\ell} \mathbb P\{ E_\sigma \mid F_\sigma\}\le \sum_{\sigma\in\{0,1\}^\ell} e^m \le \sum_{\sigma\in\{0,1\}^\ell} 2^{-n} 2^{-2\ell}     =  2^{-n}2^{-\ell},
		\end{align*}
		and hence
		\[
			\mathbb PX_n \le \sum_\ell \mathbb P\{A_G(\ell,L)\ne A_G(\ell)\}\le \sum_{\ell} 2^{-n}2^{-\ell}= 2^{-n}.
		\]
	\end{proof}

	$X_n$ is $\Sigma^0_1$ since when $L$ is defined, $A_G(\ell)$ is contained in $A_G(\ell, L)$, and $A_G(\ell)$ is $\Pi^0_1$ in $G$, which means that if the containment is proper then we can eventually enumerate (observe) this fact. Thus $\cap_n X_n$ is a $\lambda_{1,\gamma}$-ML-null set.

	$V_n$ is clearly $\Sigma^0_1$. Moreover $V_n\subseteq \Upsilon_n\cup X_n$, so $\lambda_{1,\gamma}(V_n)\le 2\cdot 2^{-n}$, hence $\cap_n V_n$ is a $\lambda_{1,\gamma}$-ML-null set.
\end{proof}

\section{Being a member of some ML-random closed set}\label{sec3}

For a real number $0\le\gamma\le 1$, the $\gamma$-weight $\wt_\gamma(C)$ of a set of strings $C\subseteq\Omega$ is defined by
\[
	\wt_\gamma(C)=\sum_{w\in C} 2^{-|w|\gamma}.
\]
We define several notions of randomness of individual reals.

A \emph{Martin-L\"of $\gamma$-test} is a uniformly $\Sigma^0_1$ sequence $(U_n)_{n<\omega}$, $U_n\subseteq\Omega$, such that for all $n$, $\wt_\gamma(U_n)\le 2^{-n}.$

A \emph{strong ML-$\gamma$-test} is a uniformly $\Sigma^0_1$ sequence $(U_n)_{n<\omega}$ such that for each $n$ and each prefix-free set of strings $V_n\subseteq U_n$, $\wt_\gamma(V_n)\le 2^{-n}$.

A real is (strongly) $\gamma$-random if it does not belong to $\inter_n [U_n]^\preceq$ for any (strong) ML-$\gamma$-test $(U_n)_{n<\omega}$.

 If $\gamma=1$ we simply say that the real, or the set of integers $\{n:x(n)=1\}$, is \emph{Martin-L\"of random (ML-random)}. For $\gamma=1$, strength makes no difference.

 For a measure $\mu$ and a real $x$, we say that $x$ is \emph{Hippocrates $\mu$-random} if for each sequence $(U_n)_{n<\omega}$ that is uniformly $\Sigma^0_1$, and where $\mu [U_n]^\preceq\le 2^{-n}$ for all $n$, we have $x\not\in \inter_n [U_n]^\preceq$.

 Let the ultrametric $\upsilon$ on $2^\omega$ be defined by
\[
	\upsilon(x,y)=2^{-\min\{n:x(n)\ne y(n)\}}.
\]
 The \emph{$\gamma$-energy} \cite{MP} of a measure $\mu$ is
\[
	I_\gamma(\mu):=\iint\frac{d\mu(b)d\mu(a)}{\upsilon(a,b)^\gamma},
\]
which in expected value ($\E$) notation can be written $I_\gamma(\mu)=\E_{(a,b)} \nu(a,b)^{-\gamma}$.
A real $x$ is \emph{Hippocrates $\gamma$-energy random} if $x$ is Hippocrates $\mu$-random with respect to some probability measure $\mu$ such that $I_\gamma(\mu)<\infty$. For background on $\gamma$-energy and related concepts the reader may consult the monographs of Falconer \cite{Falconer} and Mattila \cite{Mattila} or the on-line lecture notes of M\"orters and Peres \cite{MP}.

The terminology \emph{Hippocrates random} is supposed to remind us of the ancient medic Hippocrates, who did not consult the oracle at Delphi, but instead looked for natural causes. An almost sure property is more effective if it is possessed by all Hippocrates $\mu$-random reals rather than merely all $\mu$-random reals. In this sense Hippocratic $\mu$-randomness tests are more desirable than arbitrary $\mu$-randomness tests.

\emph{Effective Hausdorff dimension} was introduced by Lutz \cite{Lutz} and is a notion of partial randomness. For example, if the sequence $x_0x_1x_2\cdots$ is ML-random, then the sequence
\[x_00x_10x_20\cdots\] has effective Hausdorff dimension equal to $\frac{1}{2}$. Let $\Dim^1_H x$ denote the effective (or constructive) Hausdorff dimension of $x$; then we have (Reimann and Stephan \cite{RS01})
\[
	\Dim^1_H(x)=\sup\{\gamma: x\text{ is $\gamma$-random}\},
\]
which we can thus take as our definition of effective Hausdorff definition.

\begin{thm}
	[\cite{K:09}]
	\label{members}
	Each Hippocrates $\gamma$-energy random real is a $\Member_\gamma$.
\end{thm}

Here we show a partial converse:

\begin{thm}
	\label{pacific}
	Each $\Member_\gamma$ is strongly $\gamma$-random.
\end{thm}
\begin{proof}
	Let $\P=\lambda_{1,\gamma}$ and $p=2^{-\gamma}\in (\frac{1}{2},1]$. Let $i<2$ and $\sigma\in \Omega$. The probability that the concatenation $\sigma i\in G$ given that $\sigma\in G$ is by definition
	\[
		\mathbb P\{\sigma i\in G \mid \sigma\in G \} = p.
	\]
	Hence the absolute probability that $\sigma$ survives is
	\[
		\mathbb P\{\sigma\in G\}=p^{|\sigma|}=\left(2^{-\gamma}\right)^{|\sigma|} = \left(2^{-|\sigma|}\right)^\gamma.
	\]
	Let $U$ be any strong $\gamma$-test, i.e. a uniformly $\Sigma^0_1$ sequence $U_n=\{\sigma_{n,i}: i<\omega\}$, such that
for all prefix-free subsets $U_n'=\{\sigma'_{n,i}:i<\omega\}$ of $U_n$, $\text{wt}_\gamma(U_n')\le 2^{-n}$.  Let $U_n'$ be the set of all strings $\sigma$ in $U_n$ such that no prefix of $\sigma$ is in $U_n$. Clearly, $U_n'$ is prefix-free. Let
\[
	[V_n]^\preceq:=\{S: \exists i\, \sigma_{n,i}\in G\}\subseteq\{S: \exists i\, \sigma'_{n,i}\in G\}.
\]
Clearly $[V_n]^\preceq$ is uniformly $\Sigma^0_1$. To prove the inclusion: Suppose $G$ contains some $\sigma_{n,i}$. Since $G$ is a tree, it contains the shortest prefix of $\sigma_{n,i}$ that is in $U_n$, and this string is in $U_n'$. Now
\[
	\mathbb P [V_n]^\preceq \le\sum_{i\in\omega}\mathbb P\{\sigma'_{n,i}\in G\}=\sum_{i\in\omega} 2^{-|\sigma'_{n,i}|\gamma}\le 2^{-n}.
\]
Thus $V$ is a test for $\lambda_{1,\gamma}$-ML-randomness. Suppose $x$ is a $\Member_\gamma$. Let $S$ be any $\lambda_{1,\gamma}$-ML-random set with $x\in\Gamma_S$. Then $S\not\in\inter_n [V_n]^\preceq$, and so for some $n$, $\Gamma\inter [U_n]^\preceq=\nil$. Hence $x\not\in [U_n]^\preceq$. As $U$ was an arbitrary strong $\gamma$-test, this shows that $x$ is strongly $\gamma$-random.
\end{proof}

Examples of measures of finite $\gamma$-energy may be obtained from the fact that if $\Dim^1_H(x)>\gamma$ then $x$ is Hippocrates $\gamma$-energy random \cite{K:09}. If $x$ is strongly $\gamma$-random then $x$ is $\gamma$-random and so $\Dim^1_H(x)\ge\gamma$.

\begin{df}
	[\cite{Reimann}]
	A real $x$ is $h$-capacitable if $x$ is $\mu$-random for some probability measure $\mu$ with
	\[
		\mu(\sigma)\le c_{R} 2^{-h(|\sigma|)}.
	\]
	A real $x$ is $\gamma$-capacitable if $x$ is $\mu$-random for some $\mu$ with
	\[
		\mu(\sigma)\le c_{R} 2^{-|\sigma|\gamma},
	\]
\end{df}

\begin{df}
	\label{DeBartolo}
	A real $x$ is $\gamma^{+}$-capacitable if $x$ is $\mu$-random for some $\mu$ with
	\[
		\mu(\sigma)\le c_{R} 2^{-|\sigma|\gamma-f(|\sigma|)},
	\]
	where $\sum_{\sigma\in 2^{<\omega}} 2^{-f(|\sigma|)}<\infty$.
\end{df}

Definition \ref{DeBartolo} is made so that we can prove a stronger version of \cite{K:09}*{Lemma 2.5}:
\begin{lem}
	\label{Oct29-2008}
	Suppose $\mu$ is a Borel probability measure on $2^\omega$ such that for some constant $c_R$,
	\[
		\mu(\sigma)\le c_{R} 2^{-|\sigma|\gamma-f(|\sigma|)}
	\]
	for all binary strings $\sigma$, where $\sum_{\sigma\in 2^{<\omega}} 2^{-f(|\sigma|)}<\infty$. Then
	\[
		\iint\frac{d\mu(b)d\mu(a)}{\upsilon(a,b)^\gamma}<\infty.
	\]
\end{lem}
\begin{proof}
	We have
	\begin{alignat*}{2}
		\phi_\gamma(a):=& \int\frac{d\mu(b)}{\upsilon(a,b)^\gamma}=\sum_{n=0}^\infty 2^{n\gamma} \mu\{b:\upsilon(a,b)=2^{-n}\}={}\sum_{n=0}^\infty 2^{n\gamma} \mu[(a\restrict n) * (1-a(n))] \\
		\le{} &c_R \sum_{n=0}^\infty 2^{n\gamma} 2^{-(n+1)\gamma-f(n+1)}=c_R \sum_{n=0}^\infty 2^{-\gamma-f(n+1)}=\hat c <\infty.
	\end{alignat*}
	Thus
	\[
		\iint\frac{d\mu(b)d\mu(a)}{\upsilon(a,b)^\gamma}= \int \phi_\gamma(a)d\mu(a)\le \hat c<\infty.
	\]
\end{proof}

\begin{cor}
	Let $x$ be a real. We
	have the following implications:
	\begin{align*}
								& \Dim^1_H(x)>\gamma \\
		\Longrightarrow\quad 	& x\text{ is $\gamma^{+}$-capacitable} \\
		\Longrightarrow\quad 	& x\text{ is $\gamma$-energy random} \\
		\Longrightarrow\quad 	& x\text{ is Hippocrates $\gamma$-energy random} \\
		\Longrightarrow\quad	& x\text{ is a $\textsc{Member}_{\gamma}$} \\
		\Longrightarrow\quad 	& x\text{ is $\gamma$-capacitable} \quad(\Leftrightarrow\text{ $x$ is strongly $\gamma$-random})\\
		\Longrightarrow\quad 	& \Dim^1_H(x)\ge\gamma
	\end{align*}
\end{cor}
\begin{proof}
	Suppose $\Dim^{1}_{H}(x)>\gamma$. Then $x$ is $\beta$-capacitable for some $\beta>\gamma$ (see \cite{Reimann}) which immediately implies that $x$ is $\gamma^{+}$-capacitable. This implies that $x$ is $\gamma$-energy random (Lemma \ref{Oct29-2008}) which immediately implies that $x$ is Hippocrates $\gamma$-energy random. This implies by Theorem \ref{members} that $x$ is a $\Member_{\gamma}$. This in turn implies by Theorem \ref{pacific} that $x$ is strongly $\gamma$-random (which by Reimann \cite{Reimann} means that $x$ is $\gamma$-capacitable). This implies that $\Dim^1_H(x)\ge\gamma$ (see, e.g., Reimann and Stephan \cite{RS01}).
\end{proof}

\begin{thm}
	\label{dsf}
	Let $x\in 2^\omega$. We have the implications
	\[
	\Dim^1_H(x)>\gamma\quad\Longrightarrow\quad x\text{ is a }\Member_\gamma\quad\Longrightarrow\quad\Dim^1_H(x)\ge\gamma.
	\]
\end{thm}

Some non-reversals can be obtained. For instance, not every real with $\Dim^1_H(x)\ge\gamma$ is strongly $\gamma$-random \cite{RS01}; and consider Proposition \ref{hope}. We can also obtain a fairly sharp result on the minimum Kolmogorov complexity of a $\Member_{\gamma}$.

As in \cite{Reimann}, a (continuous) semimeasure is a function $\eta:2^{<\omega}\to [0,1]$ such that
\[
	\forall\sigma\, [\eta(\sigma)\ge \eta(\sigma^\frown 0)+\eta(\sigma^\frown 1)].
\]
Levin \cite{Levin} proved that there is an optimal enumerable semimeasure $\overline M$. A semimeasure is \emph{enumerable} if the set $\{(\sigma,q)\in 2^{<\omega}\times\mathbb Q: q<\eta(\sigma)\}$ is c.e. For any enumerable semimeasure $\eta$ there exists a constant $c$ such that for every $\sigma$,
\[
	\eta(\sigma)\le c \overline M(\sigma).
\]
The \emph{a priori complexity} of a string $\sigma$ is defined as $KM(\sigma):=-\log\overline M(\sigma)$.

\begin{thm}
	[Reimann \cite{Reimann}]
	Suppose $h$ is a computable order function such that for all $n$, $h(n+1)\le h(n)+1$, and $x \in 2^\omega$ is such that for all $n$, $KM(x \upto n) \ge h(n)$. Then $x$ is $h$-capacitable.
\end{thm}

From this theorem, we see that for $x$ to be $\gamma^+$-capacitable, it suffices to have $KM(x \upto n) \ge \gamma |n| + f(n)-c$, for some function $f$ such that $\sum_n 2^{-f(n)}$ converges. Furthermore, we can turn this into a condition on the prefix-free Kolmogorov
complexity $K$ rather than the \emph{a priori} complexity $KM$ by using the following theorem relating the two.

\begin{thm}
	[G\'{a}cs \cite{Gacs_82}, Uspensky and Shen \cite{Uspensky_Shen}]
	For all $\sigma \in 2^{<\omega}$,
	\[
	K(\sigma) \leq KM(\sigma)+K(|\sigma|)+O(1).
	\]
\end{thm}

Putting together these results, we obtain the following theorem.

\begin{thm}
	Suppose $x \in 2^\omega$ satisfies, for some constant $c$ and all natural numbers $n$,
	\[
	K(x \upto n) \geq \gamma n + f(n) + K(n)-c,
	\]
	where $f: \omega \to \R$ is a computable order function such that $f(n+1)-f(n)<1-\gamma$ for all but finitely many values of $n$. Then $x$ is $\gamma^+$-capacitable, and hence $x$ is a $\Member_\gamma$.
\end{thm}

\begin{pro}
	\label{hope}
	Let $\gamma=1/2$. There exists a $\gamma$-energy random real $x$ such that $\Dim^1_H(x)=\gamma$.
\end{pro}
\begin{proof}
	Consider the probability measure $\mu$ on $2^\omega$ such that $\mu([\sigma^\frown 0])=\mu([\sigma^\frown 1])$ for all $\sigma$ of even length, and such that $\mu([\sigma^\frown 0])=\mu([\sigma])$ for each $\sigma$ of odd length $f(k)=2k+1$. A computation shows that $I_\gamma(\mu)$ is finite if and only if $\gamma<1/2$. In detail,\footnote{This computation corrects a numerical error in the conference version of the present article \cite{Diamondstone.K:09}.} writing $\sigma^*$ for the neighbor string of $\sigma$, i.e.
	\[
	\sigma^{*}={\sigma\restrict_{|\sigma|-1}}^\frown (1-\sigma(|\sigma|-1)),
	\]
	we have
	\begin{align*}
		I_\gamma(\mu)=&\E_{(a,b)} \nu(a,b)^{-\gamma} = \E_a \sum_{n=0}^{\infty} 2^{n\gamma} \mu([{a\restrict_{n+1}}^*])\\
		=^{\dag} \sum_{n=0}^\infty 2^{n\gamma} &\E_a(\mu([{a\restrict_{n+1}}^*)]) = \sum_{n=0}^\infty 2^{(2k)\gamma}2^{-(k+1)}
		=\frac{1}{2} \sum_{k=0}^\infty 2^{k(2\gamma-1)}<\infty
	\end{align*}
	if $2\gamma-1<0$, i.e., $\gamma<\frac{1}{2}$. To justify the step ($\dag$), note that if $\gamma<1/2$ then for all $a$ in the support of $\mu$,
	\[
	\sum_{n=0}^\infty 2^{n\gamma} \mu([{a\restrict_{n+1}}^*]) \le \sum_{k=0}^\infty 2^{2k\gamma} 2^{-(k+1)}
	\]
	which is a finite constant, so the dominated convergence theorem applies. On the other hand, if $\gamma=1/2$ then since ``$\ge$'' always holds in ($\dag$), we have $I_\gamma(\mu)=\infty$.

	We find that $\mu$-almost all reals are $\mu$-random and have effective Hausdorff dimension exactly $1/2$. By modifying $f(k)$ slightly we can get $I_\gamma(\mu)<\infty$ for $\gamma=1/2$ while keeping the effective Hausdorff dimension of $\mu$-almost all reals equal to $1/2$. Namely, what is needed is that
	\[
	\sum_{k=0}^\infty 2^{f(k)\gamma}2^{-(k+1)}<\infty.
	\]
	This holds if $\gamma=1/2$ and $f(k)=2k-2(1-\eps)\log k$ for any $\eps>0$ since $\sum_k k^{-(1+\eps)}<\infty$. Since this $f(k)$ is asymptotically larger than $(2-\delta)k$ for any $\delta>0$, the $\mu$-random reals still have effective Hausdorff dimension $1/2$. 
\end{proof}

\begin{con}
	\label{hah}
	There is a strongly $\gamma$-random real which is not Hippocrates $\gamma$-energy random.
\end{con}

In a conference version of this article \cite{Diamondstone.K:09} we made the following conjecture.
\begin{con}
	\label{heh}
	A real $x$ is a $\Member_\gamma$ if and only if $x$ is Hippocrates $\gamma$-energy random.
\end{con}

The following considerations make Conjecture \ref{heh} seem less plausible. (The ideas here are related to selection theorems in the theory of random closed sets, which were introduced to us by David Ross at University of Hawai\textquoteleft i in May 2009.)

\begin{df}
	[Address]
	If $\Gamma$ is a closed set and $x\in \Gamma$ then the address of $x$ in $\Gamma$ is the image of $x$ under the lexicographical order preserving isomorphism between $\Gamma$ and $2^\omega$. If $y$ is the address of $x$ in $\Gamma$ then we write $x=\Gamma(y)$.
\end{df}

For example, the leftmost path in $\Gamma$ has address $0^\infty=000\ldots$ An alternative term for \emph{address} sometimes seen in the literature is \emph{signature}.

\begin{thm}
	If $x$ is Hippocrates $\gamma$-energy random then $x$ is a $\Member_\gamma$ of a closed set $\Gamma$ which is ML-random relative to the address of $x$ in $\Gamma$.
\end{thm}
\begin{proof}
	Suppose $x$ is never $\Gamma(y)$ where $\Gamma$ is ML-random relative to $y$. Then
	\[
	\{
		(\Gamma,y)\mid x=\Gamma(y)
	\}
	\subseteq
	\{
		(\Gamma,y)\mid \Gamma\in V_n^y
	\}
	\]
	for all $n$, where $V_n$ is a universal oracle test.
	Let
	\[
	U_n=
		\{
			\hat x \mid
				\{
					(\Gamma,y)\mid \hat x=\Gamma(y)
				\}
				\subseteq
				\{
					(\Gamma,y)\mid \Gamma\in V_n^y
				\}
		\}.
	\]
	The class of reals $U_n$ is $\Sigma^0_1$, as follows from compactness upon considering a no-dead-ends tree representation of $\Gamma$. As shown in an earlier paper \cite{K:09}, if $x$ is $\gamma$-energy random as witnessed by a measure $\mu$, then
	\[
	\frac{\mu(U_n)^2}{c} \le \P\{\Gamma: \Gamma\cap U_n\ne\nil\} \le (\P\times\nu_n)\{(\Gamma,y)\mid \Gamma(y)\in U_n\}
	\]
	(where $\nu_n$ almost surely picks out an element of $\Gamma\cap U_n$ if one exists)
	\[
	=(\P\times \nu_n) \bigcup_{\hat x\in U_n} \{(\Gamma,y)\mid \Gamma(y)=\hat x\} \le (\P\times\nu_n)\{(\Gamma,y)\mid \Gamma\in V_n^y\} \le 2^{-n}.
	\]
	Hence $\cap_n U_n$ is a Martin-L\"of null set for the measure $\mu$, and thus $x$ is not $\mu$-random, after all.
\end{proof}

\section{Changing the quantifier}

In this section our attention turns away from the types of reals that belong to \emph{some} ML-random closed set, and toward the types of reals can be found in \emph{all} ML-random closed sets.

Given any set $Z\subseteq\omega$ we can form the tree
\[
	T_{Z}=\{\sigma:(\forall n<|\sigma|)(Z(n)=0\to \sigma(n)=0)\},
\]
and the corresponding closed set $[T_{Z}]=\{x: (\forall\sigma\prec x)\, \sigma\in T_{Z}\}$.

\begin{lem}
	[\cite{KNe}*{Lemma 4.11}]
	\label{stronger}
	Suppose given a real number $\gamma\in (0,1)$, and $\eps>0$ such that $\gamma+\eps$ is a rational number $p/q$. If $A=[T_Z]$ with $Z=\{n: n\mod q < p\}$ then there is a probability measure $\mu$ on $A$ such that $I_\gamma(\mu)<\infty$, and such that for all for $\sigma\in\Omega$, $\mu ([\sigma])>0\iff [\sigma]\inter A\ne\nil$.
\end{lem}

Let $\Dim(B)=\Dim_H(B)$ denote the Hausdorff dimension of a set $B\subseteq 2^\omega$.

\begin{lem}
	\label{zee}
	Let $A$ be as in Lemma \ref{stronger}. For each $x\in A$, $\Dim^1_H(x)\le p/q$.
\end{lem}
\begin{proof}
	Let $\mathcal H^{p/q}_{\eps}$ denote the usual $\eps$-approximation to $p/q$-dimensional Hausdorff measure $\mathcal H^{p/q}$. Note that we can cover $A$ with $2^{mp}$ many cones $[\sigma]$ with $|\sigma|=mq$, and hence if $\eps=2^{-mq}$ then $\mathcal H^{p/q}_\eps(A)\le 2^{mp}(2^{-mq})^{p/q}=1$. As $m\to\infty$, $\eps\to 0$ and so $\mathcal H^{p/q}(A)\le 1$ and thus $\Dim_H(A)\le p/q$.
\end{proof}

\begin{thm}
	For each $\eps>0$, each $\P^*_{\gamma}$-ML-random closed set for $2^{-\gamma}=2/3$ contains a real $x$ with $\Dim^1_H(x)\le \log_2(\frac{3}{2})+\eps$.
\end{thm}

\begin{proof}
	Fix $\eps>0$. We may assume $\gamma+\eps\in\mathbb Q$. Let $A$ be as in Lemma \ref{stronger}. It follows from \cite{KNe}*{Theorem 4.10} that $\dim(A)>\gamma$. Let
	\[
	U:=\{\Gamma: (\exists n) (\forall \sigma\in G_{n})\, [\sigma]\inter A=\nil\},
	\]
	which is a $\Sigma^0_1$ class. Indeed, there are only finitely many $\sigma\in 2^{n}$ to check for a given $n$, and for our choice of set $A=[T_Z]$, $\{\sigma\in\Omega: [\sigma]\inter A=\nil\}$ is computable.

	As shown by Hawkes \cite{Hawkes}, $\P^{*}_{\gamma}(U)<1$. In fact, one way to see this is to observe that if $\P^{*}_{\gamma}(U)=1$ then $U$ would contain all $\P^{*}_{\gamma}$-ML-random closed sets $\Gamma$ (even all $\P^{*}_{\gamma}$-\emph{Kurtz random} closed sets), contradicting Theorem \ref{dsf} and the following fact (see \cite{JanPhD}):
	\begin{quote}
		each set $B$ with $\dim(B)>\gamma$ contains a real $x$ with $\Dim^1_H(x)>\gamma$.
	\end{quote}
	Let $\Gamma$ be a $\P^{*}_{\gamma}$-ML-random closed set, let $\ell$ be its leftmost path, and let $n_i$ be the $i$th zero of $\ell$ (so $\ell(n_i)=0$) and $\ell_i=(\ell\restrict n_i)^{\frown}1$. Using the notation $\sigma X=\{\sigma^{\frown} x:x\in X\}$, we have
	\[
	\Gamma=\bigcup_{i\in\omega}\ell_i\Gamma_i
	\]
	where $\Gamma_i$ is again a $\P^{*}_{\gamma}$-ML-random closed set.

	Let $U_n$ be defined by $\Gamma\in U_n$ iff $\Gamma_n\in U$. Then the events $U_n$ are mutually independent and there is a constant $u$ such that for each $n$, $\P^{*}_{\gamma}(U_{n})=u<1$. Hence $\P^{*}_{\gamma}(\cap_{n}U_{n})=\lim_{n\to\infty} u^n=0$. Because $u^n\to 0$ effectively, $\inter_n U_n$ is in fact a $\P^{*}_{\gamma}$-Martin-L\"of null set. Thus if $\Gamma$ is $\P^{*}_{\gamma}$-ML-random, then there is an $i$ for which $A\inter\Gamma_i\ne\nil$, or equivalently $(\ell_iA)\inter\Gamma\ne\nil$. Thus $\Gamma$ contains a shift of a member of $A$.

	By Lemma \ref{zee}, for each $x\in A$, $\Dim^1_H(x)\le\gamma+\eps$. Thus each $\P^{*}_{\gamma}$-ML-random closed set $\Gamma$ contains a shift $y$ of a real $x$ with $\Dim^1_H(x)\le\gamma+\eps$ and hence in fact contains $y$ with $\Dim^1_H(y)\le\gamma+\eps$.
\end{proof}

\section{Applications}

\subsection{Approximation properties}

\begin{pro}
	Let $0<\gamma<1$. If  $x$ is a real such that the function $n\mapsto x(n)$ is $f$-computably enumerable for some computable function $f$ for which ${\sum_{j<n} f(i)}{2^{-n\gamma}}$ goes effectively to zero, then $x$ is not $\gamma$-random.
\end{pro}
\begin{proof}
	Suppose $n\mapsto x(n)$ is $f$-c.e. for some such $f$, and let $F(n)=\sum_{j<n} f(n)$. Let $\alpha$ be any computable function such that $\alpha(n,i)\ne\alpha(n,i+1)$ for at most $f(n)$ many $i$ for each $n$, and $\lim_{i\to\infty}\alpha(i,n)=x(n)$. Let $c(n,j)$ be the $j$th such $i$ that is discovered for any $k<n$; so $c$ is a partial recursive function whose domain is contained in $\{(n,j):j\le F(n)\}$. For a fixed $i$, $\alpha$ defines a real $\alpha_i$ by $\alpha_i(n)=\alpha(i,n)$. Let $V_n=\{x: \exists j\le F(n)\,\, x\restrict n = \alpha_{c(n,j)}\restrict n)\}.$ Since $V_n$ is the union of at most $F(n)$ many cones $[x\restrict n]$,
	\[
	\wt_\gamma(V_n)\le\sum_{j=1}^{F(n)} 2^{-n\gamma} = {F(n)}{2^{-n\gamma}},
	\]
	which goes effectively to zero by assumption. Thus there is a computable sequence $\{n_k\}_{k\in\N}$ such that $\wt_\gamma(V_{n_k})\le 2^{-k}$. Let $U_k=V_{n_k}$. Then $U_k$ is $\Sigma^0_1$ uniformly in $k$, and $x\in\cap_k U_k$. Hence $x$ is not $\gamma$-random.
\end{proof}

\begin{cor}
	[\cite{BBCDW}]
	No member of a ML-random closed set under the Florida distribution is $f$-c.e. for any polynomial-bounded $f$.
\end{cor}
\begin{proof}
	If $f$ is polynomially bounded then clearly ${\sum_{j<n} f(i)}{2^{-n\gamma}}$ goes effectively to zero. Therefore if $x$ is $f$-c.e., $x$ is not $\gamma$-random, hence not a $\Member_\gamma$ for any $0<\gamma<1$, and thus not a member of a ML-random closed set under the Florida distribution.
\end{proof}

\subsection{Randomness for Bernoulli measures}

Our results characterize the Bernoulli measures for which random sequences are \textsc{Members}. Suppose $0\le p\le 1$. The Bernoulli measure $\mu_{p}$ on $2^{\omega}$ is uniquely defined by the properties
\begin{itemize}
	\item[(i)] $\mu_{p}\{A: A(n)=1\}=p$, and
	\item[(ii)] the events $\{A:A(n)=1\}$ are mutually independent for distinct $n\in\omega$.
\end{itemize}
An infinite binary sequence $A\in\Omega$ is ML-random for the Bernoulli measure $\mu_{p}$, or for short $\mu_{p}$-random, if for each uniformly $\Sigma^{0}_{1}(p)$ sequence of open sets $U_{n}$, $n\in\omega$, with $\mu_{p}(U_{n})\le 2^{-n}$, we have $A\not\in\cap_{n}U_{n}$. This notion was related to (the martingale characterization of) effective Hausdorff dimension by Lutz.

\begin{thm}
	[Lutz \cite{Lutz}]
	\label{Lut}
	For each $\mu_p$-random sequence $A$, the effective Hausdorff dimension of $A$ is
	\[
	H(\mu_p):=-(p\log p+\overline p\log\overline p).
	\]
\end{thm}

To find the values of $p$ for which a Bernoulli $\mu_p$-random sequence is a $\Member_{\gamma}$, note that
\[
	H(\mu_p)> \gamma\quad\Longleftrightarrow\quad p^{p} (\overline p)^{\overline p}< 2^{-\gamma}.
\]
For the value $2^{-\gamma}=\frac{2}{3}$ studied by the Florida group, a numerical calculation on the web site Wolfram Alpha yields that this inequality is equivalent to
\[
	0.140276506997464...< p< 0.859723493002535...
\]

\section{A different approach} 

\begin{df}
	A set $A\subseteq\N$ is \emph{infinitely often r.e.\ traceable} if there is a recursive function $p(n)$ such that for all $f:\N\to\N$, if $f$ is recursive in $A$ then there is a uniformly r.e.\ sequence of finite sets $E_n$, $n\in\N$, such that $E_n$ has cardinality $\le p(n)$ for each $n$; and such that for infinitely many $n$, we have $f(n)\in E_n$.

	A total function $f$ is DNR (diagonally nonrecursive) if $\neg\exists n$, $f(n)=\varphi_n(n)$, where $\varphi_n$ is the $n$th partial recursive function. (Note $f$ is total, whereas $\varphi_n$ need not be.)

	A real $A$ is \emph{Kurtz random relative to an oracle $B$} if it does not belong to any $\Pi^0_1(B)$ subset of $2^\omega$ of fair-coin measure zero.
\end{df}

\begin{thm}[see \cite{KMeS}]\label{bjoern}
	$A$ is infinitely often r.e.\ traceable iff $A$ does not compute any DNR function.
\end{thm}

\begin{thm}\label{1}\label{bjornie}
 	If $x$ is not of DNR degree then every Martin-L\"of random real is Kurtz-random relative to $x$.
\end{thm}
\begin{proof}
	Let $T$ be an $x$-recursive tree such that $[T]$ has measure zero. It suffices to show that $[T]$ is contained in $\bigcap_n U_n$, where the sets $U_n$ are (uniformly) $\Sigma^0_1$ classes of measure $\le 2^{-n}$.

	By Theorem \ref{bjoern}, since $x$ is not of DNR degree, $x$ is infinitely often r.e.\ traceable. Fix a recursive trace size bound function $p$ as in the definition of i.o.\ r.e.\ traceability. Let $V_k$ be the minimal clopen set that we can tell is a covering of $T$ by looking at the first level of $T$ where it becomes evident that the measure of $[T]$ is $<2^{-k}/p(k)$. So $[T]\subseteq V_k$.  Let $g(k)$ be (the code for) $V_k$. Since $g$ is recursive in $x$, there are infinitely many $k$ for which the value $g(k)$ is in an r.e.\ trace $S_k$ of size bounded by $p(k)$ consisting of only clopen sets $W_k$ of measure $<2^{-k}/p(k)$. In particular there is some such $k>n$, so $[T]\subseteq U_n$. Let $U_n=\bigcup_{k>n} S_k$. Then $\{U_n\}_{n\in\N}$ is clearly uniformly $\Sigma^0_1$. Moreover $\mu U_n\le \sum_{k>n} p(k) 2^{-k}/p(k)=2^{-n}$.
\end{proof}

Greenberg and Miller \cite{GM2} have shown that the converse of Theorem \ref{bjornie} also holds. Theorem \ref{bjornie} was a starting point for research toward the present paper. The idea is that if $x$ is not of DNR degree then $x$ cannot belong to a Martin-L\"of random closed set $\Gamma$, because $\Gamma$ is Kurtz random relative to $x$ and the set of paths of $\Gamma$ has measure zero. Now, as is well known a real of positive effective Hausdorff dimension computes a DNR function, and it turned out that adaptation of the work of Hawkes \cite{Hawkes} and Lyons \cite{Lyons} gave precise results in terms of effective Hausdorff dimension, as shown above. This approach was thus more powerful then the approach using Theorem \ref{bjornie}.

\newpage
\subsection*{An alternative proof of Theorem \ref{bjornie}} \footnote{We are grateful to the referee for this proof.} Assume that A is Martin-L\"of random and A is not Kurtz random relative to $X$. Since $A$ is not Kurtz random relative to $X$, there is a martingale $M\le_T X$ and a function $f\le_T X$ such that $M(A\restrict f(n)) > 4^n$ for all $n$. Now one can define a function $g\le_T X$ which outputs a string having the prefix $1^n0$ and the suffix consisting of all strings $\tau$ of length $f(n)$ for which $M(\tau) > 4^n$. There are at most $2^{f(n)}/4^n=2^{f(n)-2n}$ many of them. Since $A$ is Martin-L\"of random, there is a constant $d$ such that $K(A\restrict f(n))\ge f(n)-d$ for almost all $n$. Furthermore, there is a partial recursive function $\psi$ such that for all $n$ there is $m < 2^{f(n)-2n}$ such that 
$A\restrict f(n) = \psi(g(n),m)$. Note that by construction of $g$, the number $n$ can be computed from $g(n)$. Hence it follows that there is a constant $c$ so that 
\[
f(n) \le K(A\restrict f(n)) +d < (f(n)-2n)+ K(g(n))+c+d
\] 
for all $n$. Hence, $K(g(n))\ge n$ for almost all $n$. As $g\le_T X$, it follows from a result of Kjos-Hanssen, Merkle, and Stephan \cite{KMeS} that $X$ has DNR Turing degree.

\section*{Acknowledgments}

This material is based upon work supported by the National Science
Foundation under Grants No.\ 0652669 and 0901020.

\end{document}